\documentclass[preprint,12pt]{elsarticle}
\usepackage{amsmath}
\usepackage{amssymb}
\usepackage{tikz}
\usepackage{float}
\usetikzlibrary{arrows.meta}
\usepackage{babel}
\usepackage{hyperref}
\usepackage{epsfig}
\usepackage{graphicx}
\usepackage[cp1251]{inputenc}
\usepackage{amssymb}
\usepackage{amsthm}
\usepackage{cmap}

\usepackage{amsmath}

\biboptions{sort&compress}

\newcommand{\sysn}{\left\{\begin{array}{rcl}}
\newcommand{\sysk}{\end{array}\right.}

\newtheorem{theorem}{Theorem}[section]

\theoremstyle{example}

\newtheorem{proposition}[theorem]{Proposition}
\theoremstyle{definition}
\newtheorem{definition}[theorem]{Definition}

\newtheorem{corollary}[theorem]{Corollary}

\journal{...}

\begin{document}

\title{Some function applications of weak $\lambda$-spaces}

\author{Alexander V. Osipov}

\address{Krasovskii Institute of Mathematics and Mechanics, \\ Ural Federal
 University, Yekaterinburg, Russia}

\ead{OAB@list.ru}

\begin{abstract} In this paper it is proved that there exists (in ZFC) a separable pseudocompact weak $\lambda$-space $X$ which is not $\Delta_1$-space. It follows that there exists a space $X$ such that  $B_1(X,[0,1])$ is Choquet, while $B_1(X)$ is not Choquet.
Also it is proved that there exists a $\gamma$-set $X$  which is not weak $\lambda$-set. It follows that there exists  a zero-dimensional separable metrizable space $X$  such that $B_1(X)$ is Baire, while $B_1(X,[0,1])$ is not Choquet. These results answers the previously posed questions.
\end{abstract}


\maketitle 



\section{Definitions and notation}

Let us recall that a cover $\mathcal{U}$ of a set $X$ is called

$\bullet$ an $\omega$-cover if each finite set $F\subseteq X$ is contained in some set $U\in \mathcal{U}$;

$\bullet$ a $\gamma$-cover if for any $x\in X$ the set $\{U\in \mathcal{U}: x\not\in U\}$ is finite.

Recall that $X\subseteq 2^{\omega}$ is

$\bullet$ a $\gamma$-set if each open $\omega$-cover $\mathcal{U}$ of $X$ contains a $\gamma$-subcover of $X$ \cite{GN}.

$\bullet$ a {\it  $\lambda$-set} if every countable subset of $X$ is $G_{\delta}$ (\S 40, \cite{Kur}).

$\bullet$ a {\it weak $\lambda$-set}  if for every pair of countable disjoint sets $Y,Z\subseteq X$ there exists a $G_{\delta}$-$F_{\sigma}$ set $H$ of $X$ such that $Y\subseteq H$ and $Z\subseteq X\setminus H$ (Def. 4.1 in \cite{RRS}).

The concepts of $\gamma$-set, $\lambda$-set and weak $\lambda$-set can be considered in arbitrary topological spaces.

A topological space $X$ is a {\it $\lambda$-space} if every countable subset of $X$ is $G_{\delta}$. A topological space $X$ is a {\it weak $\lambda$-space}  if for every pair of countable disjoint sets $Y,Z\subseteq X$ there exists a $G_{\delta}$-$F_{\sigma}$ set $H$ of $X$ such that $Y\subseteq H$ and $Z\subseteq X\setminus H$.

A topological space $X$ is called a $\gamma$-space if each open $\omega$-cover $\mathcal{U}$ of $X$ contains a $\gamma$-subcover of $X$. $\gamma$-Spaces were introduced by Gerlits and Nagy in \cite{GN}.

A space $X$ is {\it Fr\'{e}chet-Urysohn} if, for any $A\subseteq X$ and any $x\in A$, there exists
a sequence $\{a_n: n\in \omega\}\subseteq A$ that converges to $x$. In \cite{GN}, it is proved that $C_p(X)$ is Fr\'{e}chet-Urysohn if and only if $X$ is a $\gamma$-space.

A family $\{A_{\alpha}: \alpha\in \kappa\}$ of subsets of a set $X$ is said to be {\it point-finite} if for every $x\in X$, $\{\alpha\in \kappa: x\in A_{\alpha}\}$ is finite.

A space $X$ has the {\it $\Delta_1$-property} \cite{KKL} if any disjoint sequence $\{A_n: n\in \omega\}$ of
countable subsets of $X$ has a point-finite open expansion, i.e., there exists a point-finite family $\{U_n: n\in \omega\}$ of open subsets of $X$ such that
$A_n\subset U_n$ for every $n\in \omega.$  The spaces with the $\Delta_1$-property are also called $\Delta_1$-spaces. It is clear that every $\lambda$-space is a $\Delta_1$-space.
If $X$ is a $\Delta_1$-space with a countable pseudocharacter then it is a $\lambda$-space (Theorem 2.19 in \cite{KKL}).

A space $X$ is called {\it $\kappa$-Fr\'{e}chet-Urysohn} if, for any open set $U\subset X$ and any point $x\in \overline{U}$, there exists a sequence $\{x_n: n\in \omega\}\subset U$ that converges to~$x$. Clearly, every Fr\'{e}chet-Urysohn space is $\kappa$-Fr\'{e}chet-Urysohn.

\medskip
A family $\{A_{\alpha}: \alpha\in \kappa\}$ of subsets of a space $X$ is said to be {\it strongly point-finite} if for every $\alpha\in \kappa$, there exists an open
set $U_{\alpha}$ of $X$ such that $A_{\alpha}\subset U_{\alpha}$ and $\{U_{\alpha}: \alpha\in \kappa\}$ is point-finite.

A space $X$ is said to have {\it property $(\kappa)$} if every pairwise
disjoint sequence of finite subsets of $X$ has a strongly point-finite subsequence.

\medskip

 In \cite{Sakai}, Sakai characterized $\kappa$-Fr\'{e}chet-Urysohn property in $C_p(X)$ showing that {\it $C_p(X)$ is $\kappa$-Fr\'{e}chet-Urysohn if and only if $X$ has the property $(\kappa)$}.

\medskip

It was also shown in \cite{KKL} that $C_p(X)$ is $\kappa$-Fr\'{e}chet-Urysohn whenever $X$ has the $\Delta_1$-property and,
under Martin's Axiom, even the Fr\'{e}chet-Urysohn property of $C_p(X)$ does not
imply that $X$ is a $\Delta_1$-space.

\medskip

 A space is {\it meager} (or {\it of the first Baire category}) if it
can be written as a countable union of closed sets with empty
interior. A topological space $X$ is {\it Baire} if the Baire
Category Theorem holds for $X$, i.e., the intersection of any
sequence of open dense subsets of $X$ is dense in $X$. Clearly, if
$X$ is a Baire space, then $X$ is not meager.

\medskip

Recall that a subset $A$ of a space $X$ is called a {\it zero-set} if $A=f^{-1}(0)$ for some $f\in C(X)=C(X,\mathbb{R})$. A subset $A$ of a space $X$ is called a {\it cozero set} if $X\setminus A$ is a zero-set.

A subset $M$ of a space $X$ is called  a {\it $Coz_{\delta}$-set} in $X$ if $M=\bigcap \{W_i: i\in \omega\}$ where $W_i$ is cozero set in $X$ for each $i\in \omega$. A subset $T$ of a space $X$ is a {\it $Zer_{\sigma}$-set} if $X\setminus T$ is a $Coz_{\delta}$-set i.e. $T=\bigcup \{F_i: i\in \omega\}$ where $F_i$ is zero set in $X$ for each $i\in \omega$.

\section{Main results}

\begin{proposition} (Theorem 4.5 in \cite{OS1}) There exists (in $ZFC$) a separable pseudocompact space $X$ with the property $(\kappa)$ that is not $\Delta_1$-space.
\end{proposition}

This space $X$ has another important property: $X$ is a weak $\lambda$-space.

\begin{theorem}\label{th2} There exists a weak $\lambda$-space that is not $\Delta_1$-space.

\end{theorem}

\begin{proof} Consider a space $X\subseteq I^{\mathfrak{c}}$ (Theorem 4.5 in \cite{OS1}) where $I=[0,1]$ such that:

(a)  $X=S\cup P$ where $S$ and $P$  are disjoint dense subsets of $I^{\mathfrak{c}}$;

(b) $S$ is a countable submaximal space;

(c) $P$ is a pseudocompact space such that every countable subset of $P$ is $C^*$-embedded.

Note that $X$ is a pseudocompact dense subset of $I^{\mathfrak{c}}$. It follows that $\beta X=I^{\mathfrak{c}}$ (see Theorem 1.7.7 in \cite{HTMT}).

Let $A$ and $B$ be countable disjoint subsets of $X$ and  $f$ be a function from $A\cup B$ into $\mathbb{R}$ such that $f(a)=0$ for each $a\in A$ and $f(b)=1$ for each $b\in B$. Since $X$ is Tychonoff and $A\cup B$ is countable, $f$ is Baire-one function on $A\cup B$.

Let us prove that $f$ is extendable to a Baire-one function on $X$.

(1) Consider the sets $A'=S\cap A$ and $B'=S\cap B$. Let $g: S\rightarrow \mathbb{R}$ be a function such that $g(s)=0$ for each $s\in S\setminus B'$ and $g(s)=1$ for each $s\in B'$. Note that $g$ is a bounded Baire-one function on $S$ and $g|_{A'\cup B'}=f$.

Since $S$ is submaximal (hence, hereditarily irresolvable) and $g$ is  a bounded Baire-one function, $g$ is fragmented (Proposition 1 in \cite{KM}).

(2) By Theorem 6.2 in \cite{KM1}, $g$ is extendable to a Baire-one function $h$ on $I^{\mathfrak{c}}$ (and, hence, on $X$).

(3) Let $f_1:=f|_{P\cap (A\cup B)}$. Since every countable subset of $P$ is closed (in $P$), $f_1$ is a continuous function on $P\cap (A\cup B)$.
Since $P\cap (A\cup B)$ is countable and $C^*$-embedded, $f_1$ is extendable to a continuous function $f_2$ on $P$.
Finally, since  $\beta P=I^{\mathfrak{c}}$ and $f_2$ is bounded (recall that $P$ is pseudocompact), $f_2$ is extendable to a continuous function $f_3$ on $I^{\mathfrak{c}}$.

(4) Let $p=h*f_3$. Then $p|_X$ is Baire-one function and $p|_{A\cup B}=f$.

Note that $A\subset p^{-1}([0,\frac{1}{3}])$ and $B\subset p^{-1}([\frac{2}{3},1])$.
Since $p|_X$ is Baire-one function, $p^{-1}([0,\frac{1}{3}])$ and $p^{-1}([\frac{2}{3},1])$ are disjoint $Coz_{\delta}$-subsets of $X$.

There are sequences $(U_n)_n$ and $(V_n)_n$ of co-zero sets of $X$ such that $p^{-1}([0,\frac{1}{3}])=\bigcap_n U_n$, $p^{-1}([\frac{2}{3},1])=\bigcap_n V_n$.

Let $\mathcal{F}=\{X\setminus U_n, X\setminus V_n : n\in\omega \}$. The countable family $\mathcal{F}$ of zero-sets is a cover of $X$. For the family $\mathcal{F}$ holds: if $F\in \mathcal{F}$, then  $A\cap F=\emptyset$ or $B\cap F=\emptyset$.

Let $\mathcal{F}=\{F_n:n\in\omega\}$. Consider $P_n=F_n\setminus \bigcup_{i<n}F_i$ for every $n\in\omega$. Let $\mathcal{P}=\{P_n:n\in\omega\}$.
Then $\mathcal{P}$  is a countable partition of $X$ consisting of $Zer_{\sigma}$-sets such that $\mathcal{P}$ is refinement of $\mathcal{F}$. Then $A\cap P=\emptyset$ or $B\cap P=\emptyset$ for $P\in \mathcal{P}$. Let $Q=\bigcup\{P\in\mathcal{P}: P\cap A\neq\emptyset\}$. Then $A\subset Q$, $B\subset X\setminus Q$ and $Q$ is a $G_{\delta}$-$F_{\sigma}$-set of $X$. It follows that $X$ is a weak $\lambda$-space.

 It remains to be noted that a pseudocompact space is a $\Delta_1$-space if and only if all countable subsets of $X$ are scattered (Theorem 2.9 in \cite{KKL}). Since $S$ is not scattered, $X$ is not $\Delta_1$-space.
\end{proof}

\begin{theorem} \label{th1} Under Martin's
Axiom there exists a $\gamma$-set that is not weak $\lambda$-set.

\end{theorem}

\begin{proof} Our space $X$ is the space $H$ from Theorem 10 of \cite{BT}. Let us recall its construction (see Example 8.4 in \cite{BG}).

Given two functions $f,g: \omega\rightarrow \omega$, we write $f\leq^* g$ if the set $\{n\in \omega: f(n)\not\leq g(n)\}$ is finite. We say that a subset $B\subseteq \omega^{\omega}$ is {\it unbounded} if
for any $y\in \omega^{\omega}$ there exists $x\in B$ such that $x\not\leq^* y.$ It is well known that the cardinal $\mathfrak{b}$  can be equivalently defined as the smallest cardinality of an
unbounded subset of $\omega^{\omega}$.

By $\omega^{\uparrow \omega}$ we denote the family of strictly increasing functions from $\omega$ to $\omega$. Note that there exists
a transfinite sequence $\{f_{\alpha}\}_{\alpha<\mathfrak{b}}$ of increasing functions $f_{\alpha} : \omega \rightarrow \omega$ such that the set $S=\{f_{\alpha}\}_{\alpha<\mathfrak{b}}$ is
unbounded in $\omega^{\omega}$ and $f_{\alpha}\leq^* f_{\beta}$ for every $\alpha<\beta$ in $\mathfrak{b}$ (see Example 8.4 in \cite{BG}).
Consider the closure $\overline{\omega}=[0, \omega]$ of $\omega$ in the ordinal $\omega+1$ endowed with the order topology (which is
compact and metrizable). Let $C\subseteq \overline{\omega}^{\omega}$ be the countable set of functions $f : \omega\rightarrow \overline{\omega}$ such that there exists
$n\in \omega$ so that $f(i) <f(j)$ for every $i <j<n$ and $f(k) = \omega$ for all $k\geq n$. Observe that the subspace
$\overline{\omega}^{\uparrow \omega}:= \omega^{\uparrow \omega}\cup C$ in $\overline{\omega}^{\omega}$ is closed, and hence compact.

According to \cite{CRZ}, it is consistent that for any $\mathfrak{b}$-scale $S=\{f_{\alpha}\}_{\alpha<\mathfrak{b}}$ the space $X:=S\cup C$ is a $\gamma$-set.

We claim that the space $X$ is not weak  $\lambda$-set.

Let $\omega=O\cup E$ where $O$ is the set of odd numbers and $E$ is the set of even numbers.
Let $C_1:=\{f\in C:$ exists $n\in O$ so that $f(i) <f(j)$ for every $i <j<n$ and $f(k) = \omega$ for all $k\geq n\}$ and
$C_2:=\{f\in C:$ exists $n\in E$ so that $f(i) <f(j)$ for every $i <j<n$ and $f(k) = \omega$ for all $k\geq n\}$. It is obvious that $C=C_1\cup C_2$, $C_1\cap C_2=\emptyset$ and $|C_1|=|C_2|=\omega.$


 Let $T\in \{O,E\}$. Note that $\pi_T: C_p(\omega, \overline{\omega}) \rightarrow C_p(T, \overline{\omega})$ such that $\pi_T(f)=f|T$ for every $f\in C_p(\omega, \overline{\omega})$ is an open map.

Assume that  $G_1, G_2$ are $G_{\delta}$ in $X$ such that $C_1\subseteq G_1$, $C_2\subseteq G_2$ and $G_1\cap G_2=\emptyset$.
Since $\pi_O$ is an open map, $\pi_O(G_1)$ is $G_{\delta}$ in $\pi_O(X)$. Note that $\pi_O(C_1)$ is a countable set of functions $f : O\rightarrow \overline{\omega}$ such that there exists $n\in O$ so that $f(i) <f(j)$ for every $i <j<n$ and $f(k) = \omega$ for all $k\geq n$. The set $\pi_O(S)$ is a  $\mathfrak{b}$-scale.
 In (\cite{BG}, see Example 8.4) it is proved that $\pi_O(C_1)$ is not $G_{\delta}$ in $\pi_O(C_1)\cup \pi_O(S)$. Moreover, there exists $h_1\in \pi_O(S)$ such that $h_1\in \pi_O(G_1)$.

Analogically, there exists $h_2\in \pi_E(S)$ such that $h_2\in \pi_E(G_1)$. Note that if $h\in S$ such that $h_1 \leq^* h|O$ and $h_2 \leq^* h|E$ then $h\in G_1$.

It remains to be noted that if instead of $G_1$ we considered $G_2$, then there exists $g_1\in \pi_O(S)$ such that $g_1\in \pi_O(G_2)$ and $g_2\in \pi_E(S)$ such that $g_2\in \pi_E(G_2)$. Also if $g\in S$ such that $g_1 \leq^* g|O$ and $g_2 \leq^* g|E$ then $g\in G_2$.

Let $p\in S$ such that $h\leq^* p$ and $g\leq^* p$. Then $p\in G_1\cap G_2$, it is contradiction. It follows that $X$ is not weak $\lambda$-set.

\end{proof}

\section{Application to function spaces}

The Baire property have dual characterization through the presence (or
absence) of a winning strategy for players in the Banach-Mazur topological game \cite{Ox}.

The game $G(X)$ begins with the first $({\bf I})$ player; Player {\bf I} chooses a nonempty open set $V_0\subseteq X$.
Then the second ({\bf II}) player chooses a nonempty open set $V_1\subseteq V_0$.
At the {\it n}th step, player {\bf I} chooses a nonempty open set $V_{2n}\subseteq V_{2n-1}$, and then player {\bf II} chooses
a nonempty open set $V_{2n+1}\subseteq V_{2n}$.
Player {\bf I} wins the game $G(X)$ if $\bigcap\limits_{n\in \omega} V_n=\emptyset$. Otherwise, player {\bf II} wins.

 A topological space $X$ is a Baire space  if and only if player {\bf I} does not have a winning strategy in the game $G(X)$

A topological space $X$ is called a {\it Choquet space}  if player {\bf II} has a winning strategy in the game $G(X)$ \cite{Ox}.

It is clear that any Choquet space is a space with the Baire property.

\medskip
For topological spaces $X$ and $Y$ , we denote by $C_p(X,Y)$ the set $C(X,Y)$ of all continuous
functions from $X$ to $Y$, endowed with the topology of pointwise convergence. A function $f : X \rightarrow Y$
from a topological space $X$ to a topological space $Y$ is called a Baire-one function  if $f$ is the pointwise limit of a sequence $\{f_n : n\in \omega\}\subset C(X, Y )$. Let $B_1(X, Y )$ denote the set of all Baire-one functions from a topological space $X$ to a topological space $Y$,
endowed with the topology of pointwise convergence. Denote by $B_1(X)=B_1(X,\mathbb{R})$.

\medskip

\medskip

A set $Y\subset X$ is called

$\bullet$ {\it $B_1$-embedded} in $X$ if for any $f\in B_1(Y)$ there exists $\tilde{f}\in B_1(X)$ such that $\tilde{f}|Y=f$;

$\bullet$ {\it $B^*_1$-embedded} in $X$ if for any bounded  $f\in B_1(Y)$ there exists $\tilde{f}\in B_1(X)$ such that $\tilde{f}|Y=f$.

\medskip

In \cite{Os2}, it is proved the following results where weak $\lambda$-space denote by $\lambda_2$-space.

\begin{theorem}(Theorem 3.8 in \cite{Os2}) For a Tychonoff space $X$ the following assertions are equivalent.

(1)   $B_1(X,[0,1])$ is a Choquet space.

(2)   Every countable subset of $X$ is $B^*_1$-embedded.

(3)  $X$ is a weak $\lambda$-space.

\end{theorem}

\begin{theorem}(Theorem 3.7 in \cite{Os2}) For a Tychonoff space $X$ the following assertions are equivalent.

(1)   $B_1(X)$ is a Choquet space.

(2)   Every countable subset of $X$ is $B_1$-embedded.

(3)  $X$ is a $\Delta_1$-space.

\end{theorem}

\medskip

For function spaces $B_1(X,[0,1])$ and $B_1(X)$ we obtain the following corollaries to Theorems \ref{th2} and \ref{th1}.

\begin{corollary}{\it  There exists (in ZFC) a separable pseudocompact space $X$ such that $B_1(X,[0,1])$ is a Choquet space, but $B_1(X)$ is not.}

\end{corollary}

This  result resolve the question (Problem 3.2) posed in \cite{Os2}.
\medskip

\medskip
In \cite{Os3} it is proved that {\it a space $X$ has the property $(\kappa)$  if and only if $B_1(X)$ is Baire}.

By Theorem 3.16 in \cite{Osip1}, $B_1(X)$ is Baire for any $\gamma$-space $X$.

\begin{corollary} {\it It is consistent that there exists a zero-dimensional separable metrizable space $X$  such that $B_1(X)$ is Baire, while $B_1(X,[0,1])$ is not Choquet.}
\end{corollary}

Note that if $B_1(X)$ is Baire then $B_1(X, [0,1])$ is Baire, too. It implies the following result.

\begin{corollary}{\it It is consistent that there exists a zero-dimensional separable metrizable space $X$  such that $B_1(X,[0,1])$ is Baire, while $B_1(X,[0,1])$ is not Choquet.}
\end{corollary}

This  result resolve the question (Problem 5.1) posed in \cite{Os2}.

\section{Spaces having property $(\kappa_2)$}

\begin{definition}(Definition 3.1 in \cite{Os2}){\it A space $X$ has the property $(\kappa_2)$,  if for every disjoint sequence $\{C_n: n\in \omega\}$
of finite subsets $C_n\subseteq X$, $C_n=A_n\cup B_n$, $A_n\cap B_n=\emptyset$, $n\in \omega$, there exists a subsequence
$\{C_{n_k}: k\in \omega\}$ such that the sets $A'=\bigcup\limits_{k=1}^{\infty} A_{n_k}$ and $B'=\bigcup\limits_{k=1}^{\infty} B_{n_k}$ are $G_{\delta}$-separated, i.e., there
exist $G_{\delta}$-sets $U$ and $V$ such that $U\cap V=\emptyset$, $A'\subseteq U$  and $B'\subseteq V.$}
\end{definition}

In (\cite{Os2}, Theorem 3.4), it is proved that {\it $B_1(X,[0,1])$ is Baire if and only if $X$ has the property $(\kappa_2)$.}

In this, section, we investigate spaces having property $(\kappa_2)$.

Note that every space having property $(\kappa)$ and every weak $\lambda$-space has the property $(\kappa_2)$.

A separable metrizable space $X$ is said to be {\it always of the first category} if every
dense in itself subset $A$ of $X$ is of the first category in itself (\cite{Kur},  p. 516). For a subspace $X$ of a separable metrizable
space $Y$, $X$ is always of the first category iff for every perfect set $P$ of $Y$ (i.e. $P$ is dense in itself and closed in $Y$),
the set $P\cap X$ is of the first category in $P$ (\cite{Kur}, Theorem 1, p. 516). The latter condition is usually said to be {\it perfectly
meager} (\cite{Miller}, p. 213). If $Y$ is a Polish space then every weak $\lambda$-set $X$ of $Y$ is perfectly meager (Theorem 4.4 in \cite{RRS}).
Every separable metrizable space having property $(\kappa)$ is always of the first category (Theorem 3.2 in \cite{Sakai}).

\begin{theorem}{\label{th4}} Every separable metrizable space having property $(\kappa_2)$ is always of the first category.
\end{theorem}

\begin{proof} Let $X$ be a space having property $(\kappa_2)$, and let $A$ be a dense in itself subset of $X$. We show that $A$ is of the first
category in itself. Since $A$ is a dense in itself separable metrizable space, $A$ contains a dense subspace $D$ homeomorphic
to $\mathbb{Q}$ (\cite{Eng}, 6.2.A(e), p. 371). By Lemma 3.1 in \cite{Sakai}, there exists a pairwise disjoint family $\{F_n\}_{n\in \omega}$ of finite subsets
of $D$ such that for every infinite $I\subset\omega$, $\bigcup\limits_{n\in I} F_n$ is dense in $D$.
 Note that  $\bigcup\limits_{n\in I} F_n$ is also homeomorphic to $\mathbb{Q}$.
Let $C_1=F_1\cup F_2$, $C_n=F_{2n-1}\cup F_{2n}$. Since $X$ is a space having property $(\kappa_2)$, there exists a subsequence $\{C_{n_{k}}: k\in \omega\}$ such that the sets $A'=\bigcup\limits_{k=1}^{\infty} F_{2n_{k}-1}$ and $B'=\bigcup\limits_{k=1}^{\infty} F_{2n_{k}}$ are $G_{\delta}$-separated, i.e., there
exist $G_{\delta}$-sets $U$ and $V$ such that $U\cap V=\emptyset$, $A'\subseteq U$  and $B'\subseteq V.$  Since $A'$ and $B'$ are dense subsets of $A$,  $U$ and $V$ are dense subsets of $A$, too. If we assume that $A$ is of the second category in itself, then $U\cap V\not=\emptyset$, it is a contradiction.
\end{proof}

A subset $X$ of $\mathbb{R}$ is said to be a {\it Luzin set} if it is uncountable and for every set $A$ of the first category in $\mathbb{R}$, $A\cap X$ is
countable. Assuming the continuum hypothesis, there exists a Luzin set, for example refer to \cite{Miller} (Theorem 2.1).

\begin{corollary} {\it Let $X$ be $\mathbb{R}$, $\mathbb{P}$, $\mathbb{C}$ or a Luzin set. Then $B_1(X,[0,1])$ is meager.}
\end{corollary}

\begin{proof} Note that in any case, being always of the first category is not satisfied. By Theorem 3.4 in \cite{Os2}, {\it $B_1(X,[0,1])$ is meager if and only if $X$ does not have the property  $(\kappa_2)$.}
\end{proof}

\begin{corollary} {\it Let $Y$ be a Polish space. Then every subspace $X$ having property $(\kappa_2)$ of $Y$ is perfectly meager.}
\end{corollary}

In \cite{FM}, forcing was used to construct a consistent example of a $Q$-set $X$ which is concentrated on a countable dense subset $F$ of $2^{\omega}$ ($F$ is in the ground model). The proof actually shows that $X$ is concentrated in every dense subset $F'$ of $F$ which is in the ground model. Therefore, the space $Y=X\cup F$ is perfectly meager and
and it does not have the property  $(\kappa_2)$.

\begin{proposition} It is consistent that there exists a perfectly meager set which does not have the property  $(\kappa_2)$.

\end{proposition}

A space is said to be {\it hereditarily Baire} if every closed subspace is Baire. If a space is hereditarily Baire, then it
does not have any $G_{\delta}$-set homeomorphic to $\mathbb{Q}$.


\begin{theorem}\label{th5} Let $X$ be a first countable hereditarily Baire regular space. Then the following assertions are equivalent:

1. $X$ has property $(\kappa_2)$;

2. $X$ has property $(\kappa)$;

3. $X$ is scattered;

4. $X$ is $\Delta_1$-space;

5. $X$ is $\lambda$-space;

6. Every countable subsets of $X$ is scattered.
\end{theorem}

\begin{proof}

$(1)\Rightarrow (3)$. Suppose that $X$ is not scattered. Let $Y$ be a closed dense in itself subset of $X$. Since $X$ is first countable,
$Y$ contains a countable dense in itself subset $Z$. Since $Z$ is a countable, dense in itself and metrizable space, it is
homeomorphic to $\mathbb{Q}$ (\cite{Eng}, 6.2.A(d), p. 370). Then $\overline{Z}$ is dense in itself. By Theorem \ref{th4}, $\overline{Z}$ is the first category. Since $X$ is hereditarily Baire, $\overline{Z}$ is Baire, it is a contradiction.

$(3)\Rightarrow (2)$. By Corollary 3.8 in \cite{Sakai}, every scattered space has property $(\kappa)$.

$(2)\Rightarrow (1)$. It is trivial.

$(3)\Rightarrow (4)$. By Theorem 2.5 in \cite{KKL}, every regular scattered space belong to the class $\Delta_2$ and hence $X$ belong to the class $\Delta_1$.

$(4)\Leftrightarrow (5)$. By Theorem 2.19 in \cite{KKL}, if $X$ is a space such that every point of $X$ is $G_{\delta}$ then $X$ is a $\Delta_1$-space if and only if $X$ is a $\lambda$-space.

$(4)\Leftrightarrow (6)$. Since $X$ is a regular hereditary Baire space, by Proposition 2.5 in \cite{KKL}, every countable subset of X is scattered.

$(6)\Leftrightarrow (2)$. By Corollary 3.3 in \cite{T}, if all countable subsets of a space $X$ are scattered, then X has the property
$(\kappa)$.

\end{proof}

\begin{corollary} {\it Let $X$ be the Michael line or the Sorgenfrey line. Then $B_1(X,[0,1])$ is  meager.}
\end{corollary}

Recall that a topological space $X$ is sequential if every sequentially closed subset $A$ of $X$ is closed in $X.$

Applying Theorem 3.16 and Proposition 3.12 in \cite{T}, we get a following result.

\begin{theorem} Let $X$ be a sequential hereditarily Baire regular space. Then the following assertions are equivalent:

1. $X$ has property $(\kappa_2)$;

2. $X$ has property $(\kappa)$;

3. $X$ is scattered;

4. $X$ is $\Delta_1$-space;

5. Every countable subsets of $X$ is scattered.
\end{theorem}

\medskip

The following Figure \ref{fig:diagram1} shows the relationship between the study spaces, examples and Question 1 (Q.1).

\begin{figure}[H]
\begin{center}
\makebox[0pt]{\begin{tikzpicture}[xscale=7.7, yscale=2, >=Stealth]

    \node (TL) at (0,2) {$\Delta_1\text{--spaces}$};
    \node (TR) at (1,2) {$\text{Spaces with }(\kappa)$};

    \node (ML) at (0,1) {$\text{Weak }\lambda\text{--spaces}$};
    \node (MR) at (1,1) {$\text{Spaces with }(\kappa_2)$};

    \node (BR) at (1,0) {$\text{Perfectly meager}$};

    \draw[line width=0.8pt, double, double distance=2pt, ->] (0.18,2) -- (0.78,2);

    \draw[line width=0.8pt, double, double distance=2pt, ->] (0.23,1) -- (0.76,1);

    \draw[line width=0.8pt, double, double distance=2pt, ->] (0,1.8) -- (0,1.2);

    \draw[line width=0.8pt, double, double distance=2pt, ->] (0.95,1.8) -- (0.95,1.2);

    \draw[line width=0.8pt, double, double distance=2pt, ->] (0.95,0.8) -- (0.95,0.2);


    \draw[dashed, ->] (0.05, 1.2) -- node[pos=0.5] {$\times$}
        node[pos=0.5, right, xshift=-1mm, font=\footnotesize] {\begin{tabular}{l}Th. 2.2\\(ZFC)\end{tabular}}
        (0.05, 1.8);

      \draw[dashed, ->] (0.8, 1.8) -- node[pos=0.5] {$\times$}
        node[pos=0.75, above right, xshift=15mm, yshift=-6mm, font=\footnotesize] {\begin{tabular}{l}Th. 2.3\\(MA)\end{tabular}}
        (0.2, 1.2);

    \draw[dashed, ->] (1.0, 1.2) -- node[pos=0.5] {$\times$} (1.0, 1.8);
    \node[font=\small] at (1.1, 1.5) {Q.1};


    \draw[dashed, ->] (0.8, 2.1) .. controls (0.5, 2.4) .. node[pos=0.5] {$\times$}
        node[pos=0.5, above, yshift=1mm, font=\footnotesize] {\begin{tabular}{c}Pr. 2.1\\(ZFC)\end{tabular}}
        (0.2, 2.1);

    \draw[dashed, ->] (0.8, 0.9) .. controls (0.5, 0.6) .. node[pos=0.5] {$\times$}
        node[pos=0.5, below, yshift=-1mm, font=\footnotesize] {\begin{tabular}{c}Th. 2.3\\(MA)\end{tabular}}
        (0.2, 0.9);


    \draw[dashed, ->] (1, 0.2) -- node[pos=0.5] {$\times$}
        node[pos=0.5, right, xshift=0.1mm, font=\footnotesize] {\begin{tabular}{l}Pr. 4.5\\(ground model)\end{tabular}}
        (1, 0.8);
\end{tikzpicture}}
\end{center}
\caption{Relationship between the
study spaces}
\label{fig:diagram1}
\end{figure}
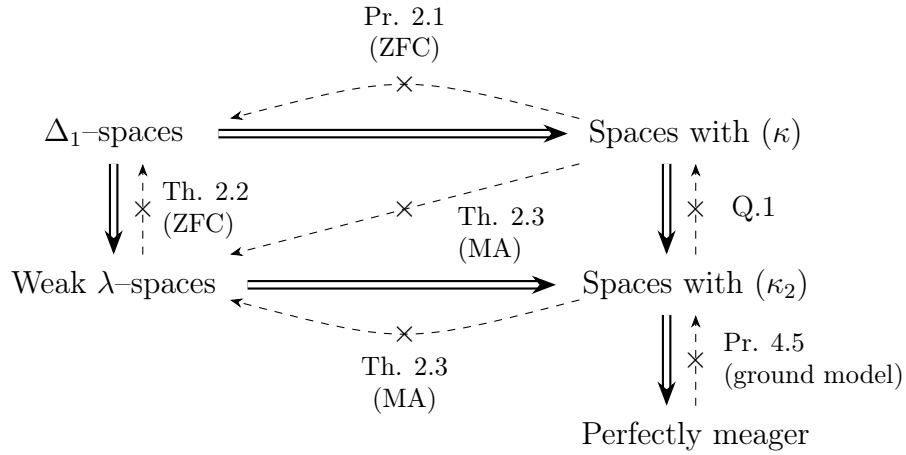

The correspondence of properties in Figure \ref{fig:diagram1} can be represented through function characteristics (see Figure \ref{fig:diagram2}).

\begin{figure}[H]
\begin{center}
\begin{tikzpicture}[xscale=7.7, yscale=2, >=Stealth]

    \node (TL) at (0,1) {$B_1(X)\text{ is Choquet}$};
    \node (TR) at (1,1) {$B_1(X)\text{ is Baire}$};
    \node (BL) at (0,0) {$B_1(X, [0,1])\text{ is Choquet}$};
    \node (BR) at (1,0) {$B_1(X, [0,1])\text{ is Baire}$};

    \draw[line width=0.8pt, double, double distance=2pt, ->] (0.27,1) -- (0.77,1);

    \draw[line width=0.8pt, double, double distance=2pt, ->] (-0.1,0.8) -- (-0.1,0.2);

    \draw[line width=0.8pt, double, double distance=2pt, ->] (0.3,0) -- (0.77,0);

    \draw[line width=0.8pt, double, double distance=2.5pt, ->] (0.9,0.8) -- (0.9,0.2);


     \draw[dashed, ->] (-0.05, 0.2) -- node[pos=0.5] {$\times$}
        node[pos=0.5, right, xshift=-1mm, font=\footnotesize] {\begin{tabular}{c}Th. 2.2\\(ZFC)\end{tabular}}
        (-0.05, 0.8);

    \draw[dashed, ->] (0.95, 0.2) -- node[pos=0.5] {$\times$} (0.95, 0.8);
    \node[font=\small] at (1.05, 0.5) {Q.1};


    \draw[dashed, ->] (0.8, 1.1) .. controls (0.5, 1.4) .. node[pos=0.5] {$\times$} node[pos=0.5, above, yshift=1mm, font=\footnotesize] {\begin{tabular}{c}Pr. 2.1\\(ZFC)\end{tabular}} (0.2, 1.1);

    \draw[dashed, ->] (0.8, -0.1) .. controls (0.5, -0.4) .. node[pos=0.5] {$\times$} node[pos=0.5, below, yshift=-1mm, font=\footnotesize] {\begin{tabular}{c}Th. 2.3\\(MA)\end{tabular}} (0.2, -0.1);


     \draw[dashed, ->] (0.75, 0.85) -- node[pos=0.5] {$\times$}
        node[pos=0.4, below right, xshift=-1mm, yshift=1mm, font=\footnotesize] {\begin{tabular}{c}Th. 2.3\\(MA)\end{tabular}}
        (0.25, 0.15);

\end{tikzpicture}
\end{center}
\caption{Function relationship between the
study spaces}
\label{fig:diagram2}
\end{figure}
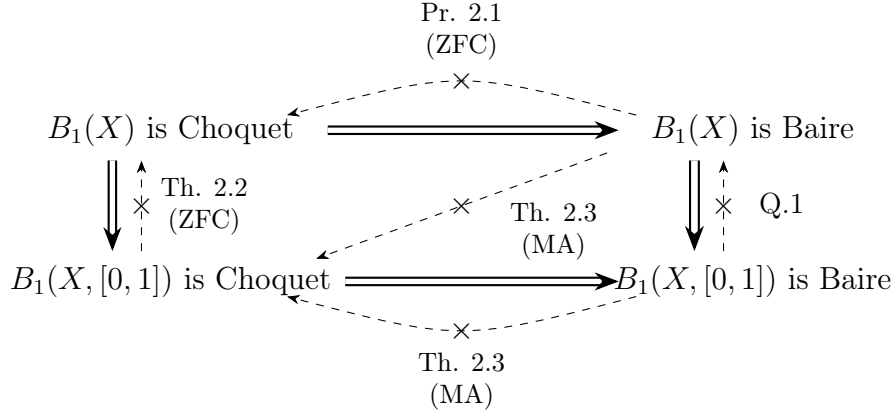



\section{Open questions}

{\bf Question 1.} {\it Is there a (separable metrizable) space with the property $(\kappa_2)$ but without the property $(\kappa)$?}

It is clear that this question is equivalent in a function context to the following: {\it Is there a (separable metrizable) space $X$ such that $B_1(X,[0,1])$ is Baire, but $B_1(X)$ is meager? }

\medskip

{\bf Question 2.} {\it Is there a separable metrizable space $X$ such that $B_1(X,[0,1])$ is Choquet, but $B_1(X)$ is not Choquet?}

Note  that this question is equivalent to the following (Question 4.6 in \cite{RRS}): {\it  Is there, at least consistently, a weak $\lambda$-set that is not a $\lambda$-set?}
\medskip

Under Martin's Axiom there exists a space with the property $(\kappa)$ (and hence with the property $(\kappa_2)$)  which is not weak $\lambda$-set (Theorem \ref{th1}).

\medskip

{\bf Question 3.} {\it Is there in $ZFC$ a space $X$ with the property $(\kappa_2)$ (or $(\kappa)$) which   is not weak $\lambda$-space?}

\medskip

{\bf Question 4.} {\it Is there, at least consistently, a weak $\lambda$-space which without the property $(\kappa)$?}

This question is equivalent in a function context to the following: {\it Is there a space $X$ such that $B_1(X,[0,1])$ is Choquet, but $B_1(X)$ is meager? }


\end{document}